\documentclass[12pt,a4paper, reqno]{amsart}
\usepackage[utf8]{inputenc}
\usepackage{mathptmx}
\usepackage{amssymb}
\usepackage{hyperref}
\usepackage{mathrsfs}
\usepackage[mathscr]{euscript}
\usepackage{booktabs}

\usepackage{todonotes}

\usepackage[ignoreunlbld,norefs,nocites,nomsgs]{refcheck}

\usepackage{pgf, tikz}

\usepackage{fancyvrb} \RecustomVerbatimEnvironment{Verbatim}{Verbatim}
{xleftmargin=15pt, frame=single, fontsize=\small}

\usepackage{color}

\definecolor{light}{gray}{0.9}
\definecolor{medium}{gray}{0.8}

\newtheorem{theorem}{Theorem}
\newtheorem{lemma}[theorem]{Lemma}
\newtheorem{corollary}[theorem]{Corollary}
\newtheorem{proposition}[theorem]{Proposition}

\theoremstyle{definition}
\newtheorem{remark}[theorem]{Remark}

\numberwithin{equation}{section}

\def\ZZ{{\mathbb Z}}

\def\RR{{\mathbb R}}
\def\QQ{{\mathbb Q}}
\def\11{{\mathbb 1}}

\def\cS{{\mathscr S}}
\def\cH{{\mathscr H}}
\def\cC{{\mathscr C}}

\def\Hilb{\operatorname{Hilb}}

\def\cone{\operatorname{cone}}

\def\Aut{\operatorname{Aut}}
\def\rec{\operatorname{rec}}

\def\iso{\cong}
\def\dirsum{\oplus}

\def\mon{\operatorname{mon}}
\def\type{\operatorname{type}}
\def\etype{\operatorname{etype}}

\let\epsilon=\varepsilon
\let\tilde=\widetilde

\allowdisplaybreaks

\begin{document}
	

\title{Automorphism groups and normal forms in Normaliz}

\author{Winfried Bruns}
\address{Universität Osnabrück, Institut für Mathematik, 49069 Osnabrück, Germany}
\email{wbruns@uos.de}
\dedicatory{For Jürgen on his 80th birthday}
\date{}


\begin{abstract}
We discuss the computation of automorphism groups and normal forms of cones and polyhedra in Normaliz, and indicate its implementation via nauty. The types of automorphisms include integral, rational, Euclidean and combinatorial, as well as algebraic for polytopes defined over real algebraic number fields. Examples treated in detail are the icosahedron and linear ordering polytopes whose Euclidean automorphism groups are determined.
\end{abstract}

\maketitle

\section{Introduction}

In this note we discuss the computation of various types of automorphism groups of cones and polyhedra in the software package Normaliz \cite{Nmz}. Automorphism groups are a classical theme, especially for regular polytopes $P$ and polytopes derived from them. For this class one wants to understand the groups of rigid motions that map $P$ to itself. Since the extension of Normaliz to real algebraic number fields, regular polytopes can be defined, and their Euclidean automorphisms are computable.

For rational polytopes and cones we are mainly interested in automorphisms defined over $\ZZ$. We have computed the integral automorphism groups for polytopes in \cite{ICP} and used isomorphism types in the experiments of the project \cite{ToricExp}.	While these computations were based on our own routines, Normaliz now uses nauty \cite{nauty} as its engine for the computation of automorphism groups, raising the level of computability and saving computation time by several orders of magnitude. In addition to our own approach, Normaliz also uses the method introduced by Bremner et al. \cite{Bremner} for the computation of automorphism groups. 

Automorphism groups are not only interesting for their own sake: they can be exploited in the computation of distinguished lattice points like  the Hilbert basis, or enumerative invariants, for example the (lattice normalized) volume. A special version of Normaliz has been used to check Wilf's conjecture for numerical semigroups of multiplicities $\le 18$; see Brunset al.\ \cite{Wilf}. This computation would have been impossible without the exploitation of the group of integral automorphisms of the so-called Kunz polyhedra. 

One of the Normaliz algorithms for polytope volumes uses a descent in the face lattice; see Bruns and Ichim \cite{Descent}. In version 3.9.0 it has now been refined by the identification of isomorphic faces. This requires the computation of integral isomorphism types. Further applications of integral automorphisms and isomorphism types are in preparation.

All computations of automorphism groups and isomorphism types must be reduced to the computation of permutations of finitely many vectors that preserve symmetric bilinear forms defined on the ambient vector space or the natural evaluation of linear forms on vectors. Therefore one needs a distinguished set of vectors and/or linear forms that are permuted by the automorphisms under consideration. These are not always available, or perhaps only after an extension of the field of definition. We will explain this obstruction when it comes up.

One consequence of the necessity to work with a finite set of vectors or linear forms is that we must pass from a non-pointed cone to its pointed quotient and from a polyhedron to its quotient modulo its maximal linear subspace. This passage, together with other simplifying assumptions is explained in Section \ref{prelim}. There we also discuss the passage from a polyhedron to the cone over it.

Section \ref{int_iso_auto} is devoted to integral automorphism groups and isomorphism types, the most interesting for Normaliz. For enumerative applications one must restrict oneself to isomorphisms that respect the degree of vectors, and the passage from a polyhedron to the cone over it endows the latter with another (possibly additional) linear form that allows to go backwards from the cone to the polyhedron.

In Section \ref{other} we outline the computation of rational, algebraic, Euclidean and combinatorial automorphisms, and mention groups of automorphisms whose computation is based on the raw input to Normaliz and does not always yield intrinsic data of the cone or polyhedron defined by the input.

Section \ref{impl} explains the preparation of the input to nauty and lists the computation goals of Normaliz that are available for automorphism groups. Finally, in Section \ref{examples}  we illustrate the computations of Normaliz by two classical examples, the icosahedron and the linear ordering polytopes.  The combinatorial automorphism groups of the linear ordering polytopes have been determined by Fiorini \cite{Fio}. Inspired by Normaliz computations, we determine the Euclidean automorphism groups.

\section{Preliminaries}\label{prelim}

For the basic terminology of discrete convex geometry we refer the reader to Bruns and Gubeladze \cite{BrGu}. In this section we want to fix some basic assumptions that computationally amount to coordinate transformations and that will be assumed in the following to simplify the formalities. Almost all data that will appear are intrinsic and do not depend on the choice of coordinates.

\subsection{Cones}
Let $C$ be a cone in $\RR^d$. The first assumption is that $C$ is full dimensional, i.e., $\dim C = d$. While this restriction may seem completely irrelevant, it is not: the definition of the \emph{dual cone} $C^*=\{\lambda \in (\RR^d)^*: \lambda(x) \ge 0 \text{ for all } x\in C  \}$ depends on it. For the passage to the full dimensional case one simply chooses coordinates in the vector subspace $\RR C$ generated by $C$.

Let $C_0=\{x \in C: -x\in C  \}$. This set of invertible elements of $C$ is a vector subspace of $\RR^d$. The cone $C/C_0$ is the image of $C$ under the natural projection $\RR^d\to \RR^d/C_0$. It is \emph{pointed}, i.e., 0 is its only invertible element. Moreover one has an isomorphism
$$
C = C_0\dirsum C/C_0
$$
of cones. Roughly speaking, the automorphism groups that are our main focus split in the same way. The passage from $C$ to $C/C_0$ is done by Normaliz. It is inevitable for finiteness results.

The passage to a full dimensional pointed cone only concerns the structure as an object of real convex geometry. Integral data, for example Hilbert bases or Hilbert series enumerating lattice points, are defined with respect to a sublattice $L$ of $\ZZ^d$ such that $C$ is generated as a cone by elements of $L$. By a further coordinate transformation we pass to a $\ZZ$-basis of $L$ which (if $C$ is full dimensional) is also a basis of the vector space $\RR^d$. After this transformation we can assume that $L=\ZZ^d$.

To sum up: we will assume that a cone $C\subset \RR^d$ is full dimensional and pointed. Under this assumption $C^*$ is pointed and full dimensional as well, and  $C\iso C^{**}$. Moreover, $C$ and $C^*$ are generated by their extreme rays. 

The lattice of reference is $\ZZ^d$, should it be relevant. The lattice structure defines natural choices of generators for $C$ and $C^*$. Each extreme ray of $C$ contains a unique integral point with coprime coordinates that we call an \emph{extreme integral generator} of $C$. The extreme integral generators of $C^*$ are called the \emph{support forms} of $C$. This terminology is justified since the support forms $\sigma_1,\dots,\sigma_s$ of $C$ define support hyperplanes $H_i=\{x\in \RR^d: \sigma_:i(x)=0\}$ of $C$, and $C$ is the irredundant  intersection of the linear halfspaces $H_i^+=\{x\in \RR^d: \sigma_i(x)\ge 0\}$. The intersections $C\cap H_i$ are the facets of $C$.
 
\subsection{Polyhedra}
By definition  a polyhedron $P\subset \RR^d$ is the intersection of finitely many \emph{affine} halfspaces. Theoretically and computationally one associates a cone $\cC(P)$ with $P$, the \emph{cone over $P$}, defined as the closure of the set $\RR_+(P\times \{1\})$ in $\RR^{d+1}$ if $P\neq\emptyset$, and  $\cC(\emptyset)=\{0\}$. The \emph{dehomogenizing hyperplane} is $\cH(P)=\{x\in\RR^{d+1}:x_{d+1}=1\}$.

The passage from $P$ to $\cC(P)$ is called \emph{homogenization}. Normaliz uses the term \emph{dehomogenization} for the linear form $\delta$ which defines $\cH(P)$ by $\cH(P)=\{x\in\RR^{d+1}:\delta(x)=1\}$. So far $\delta(x)=x_{d+1}$, but for flexibility it is necessary to allow any nonzero $\delta\in (\RR^{d+1})^*$. Clearly, as soon as integrality comes into play, $\delta$ must take integer values on $\ZZ^{d+1}$.

There is a second cone defined by $P$, its \emph{recession cone}, namely $\rec P=\{x\in \cC(P): \delta(x) = 0 \}$. One has $\rec P=\{0\}$ if and only if $P$ is a polytope, ,i.e., a bounded polyhedron. The vectors in $x\in \rec P$ satisfy the condition that $y+x\in P$ for all $y\in P$. The condition is also necessary for $x$ to be in $\rec P$ if $P\neq\emptyset$.

This allows us to compute automorphism groups and isomorphism types of polyhedra in terms of $\cC(P)$: for the automorphism group we select the automorphisms that map $\cH(P)$ into itself, and for isomorphism types the hyperplane $\cH(P)$ must be encoded in the normal form.

The coordinate transformations mentioned above are applied to $\cC(P)$ in order to reach the full dimension for $\cC(P)$ and to pass to a pointed quotient. This includes the transformation of the dehomogenization $\delta$ which need no longer be a coordinate function after the transformation.

\section{Integral isomorphisms and automorphisms}\label{int_iso_auto}

\subsection{Rational cones}
A rational cone $C\subset \RR^d$ is generated by finitely many vectors with integral coordinates. The intersection $\mon C=C\cap \ZZ^d$ is a finitely generated monoid by Gordan's lemma. It has a unique minimal system of generators (if $C$ is pointed), called the \emph{Hilbert basis}. For all this see \cite{BrGu}. A $\ZZ$-isomorphism (or integral isomorphism) is represented by a matrix with entries in $\ZZ$ whose inverse has integral entries  as well.

\begin{theorem}\label{basic}
Let $C$ and $D$ be cones in $\RR^d$. Then the following are equivalent for an $\RR$-automorphism $\phi$ of $\RR^d$:
\begin{enumerate}
\item $\phi$ restricts to a $\ZZ$-isomorphism of $C$ and $D$;
\item $\phi$  maps $\Hilb C$ onto $\Hilb D$;
\item $\phi$  restricts to an isomorphism of the monoids $\mon C$ and $\mon D$.
\end{enumerate}
In particular there are only finitely many $\ZZ$-isomorphisms of $C$ and $D$.
\end{theorem}

\begin{proof}
A $\ZZ$-isomorphism of $C$ and $D$ maps $\ZZ^d$ bijectively onto itself, and $C$ bijectively onto $D$. Therefore it maps $\mon C$ bijectively onto $\mon D$. Since $\phi$ is additive, it is an isomorphism of the two monoids.

The Hilbert bases of $C$ and $D$ are uniquely determined by algebraic conditions on their elements: they consist of the irreducible elements in $\mon C$ and $\mon D$ respectively.

Both Hilbert bases generate $\RR^d$ as a vector space. A linear endomorphism of $\RR^d$ that contains a system of generators in its image is automatically bijective. Moreover, the cone generated by $\Hilb C$ is mapped onto the cone generated by $\Hilb D$.

The finiteness of the set of isomorphisms follows immediately from (2).
\end{proof}

To simplify language we will identify $\phi$ with its pertaining restrictions in the situation of Theorem \ref{basic}.

\begin{corollary}\label{basic_aut}
With the notation of Theorem \ref{basic} the following are equivalent:
\begin{enumerate}
\item $\phi$ is a $\ZZ$-automorphism of $C$;
\item $\phi$  maps $\Hilb C$ onto itself;
\item $\phi$  is an automorphism of the monoid $\mon C$.
\end{enumerate}
In particular there are only finitely many $\ZZ$-automorphisms of $C$.
\end{corollary}

The basic computational tasks are
\begin{enumerate}
\item deciding whether $C$ and $D$ are $\ZZ$-isomorphic;
\item computing  the group $\Aut_\ZZ C$.
\end{enumerate}
For these related tasks it is useful to bring duality into play. Let $\phi:V\to W$ be a linear map of vector spaces. The dual $\phi^*:W^*\to V^*$ of $\phi$ is defined by $(\phi^*(\lambda))(x) = \lambda(\phi(x))$ for $x\in  V$, $\lambda\in W^*$. Finite dimensional vector spaces are reflexive: the bidual evaluation $\langle \lambda, x\rangle = \lambda(x)$ for $\lambda\in V^*$, $x\in V$ induces a natural linear map $V\to V^{**}$, which is a functorial isomorphism if $V$ has finite dimension. In particular one can identify $\phi^{**}$ and $\phi$ if $\phi$ is a homomorphism of finite dimensional vector spaces. For isomorphisms $\phi$ it is convenient to consider
$$
\phi^\vee = (\phi^*)^{-1} = (\phi^{-1})^*.
$$
Note that the pair $(\phi,\phi^\vee)$ preserves the bilinear evaluation of $\RR^d\times (\RR^d)^*$:  for $x\in\RR^d$ and $\lambda\in (\RR^d)^*$ one has
$$
\langle \phi(x),\phi^\vee(\lambda)\rangle=\phi^\vee(\lambda)(\phi(x)) = \lambda(\phi^{-1}(\phi(x))) = \lambda(x) = \langle x, \lambda\rangle.
$$

All this carries over to finite dimensional cones. Recall that the dual cone of the cone $C\subset \RR^d$ is
$$
C^* = \{\lambda\in (\RR^d)^*: \lambda(x) \ge 0 \text{ for all }x\in C\},
$$
and that one can naturally identify $C$ and $C^{**}$. If $\phi:C\to D$ is an isomorphism, then  $\phi^\vee: C^*\to D^*$ is an isomorphism as well. 

The following theorem can help in the computation of automorphism groups, as we will see below. The proof is easy and can be left to the reader.

\begin{theorem}\label{dual_aut}
The map $\vphantom{C}^\vee:\Aut_\ZZ C \to \Aut_\ZZ C^*$ is an isomorphism.
\end{theorem}

The decision whether cones $C$ and $D$ are isomorphic can be based on the comparison of normal forms. For the definition of the normal form we use the support forms $\sigma_1,\dots,\sigma_s$ of $C$. They define the \emph{standard embedding} of $C$: it is the map
$$
\epsilon_C: C\to \RR^s,\qquad \epsilon_C(x) =(\sigma_1(x),\dots,\sigma_s(x)).
$$
At this point we must break our simplifying assumptions: in general $\epsilon_C(C)$ is not full dimensional, and its lattice of reference is $\epsilon_C(\ZZ^d)$. (In general one has $\epsilon_C(\ZZ^d)\neq \epsilon_C(C)\cap \ZZ^s$. The quotient $\ZZ^s/\epsilon_C(\ZZ^d)$ is the class group of $\mon C$; see \cite[4.62]{BrGu}.)

\begin{theorem}\label{standard_emb}
The cones $C$ and $D$ in $\RR^d$ are $\ZZ$-isomorphic if they have the same number $s$ of facets,  $\epsilon_c(C)=\epsilon_D(D)$ and $\epsilon_C(\ZZ^d)=\epsilon_D(\ZZ^d)$, up to a permutation of the coordinates of $\RR^s$.
\end{theorem}

\begin{proof}
By construction one has $\ZZ$-automorphisms $C\cong \epsilon_C(C)$ and $D\cong \epsilon_D(D)$, which proves the implication $\Longleftarrow$. For the converse implication let $\phi:C\to D$ be a $\ZZ$-automorph\-ism.  Then $\phi^\vee(\sigma_1), \dots, \phi^\vee(\sigma_s)$ are the support forms of $D$ if $\sigma_1,\dots\sigma_s$ are the support forms of $C$. This implies $\epsilon_D(\phi(x))=\epsilon_C(x)$ for all $x\in C$, including $x\in \mon C$ which is mapped isomorphically onto $\mon D$. Hence $\epsilon_C(C)=\epsilon_D(D)$ and $\epsilon_C(\ZZ^d)= \epsilon_D(\ZZ^d)$ (if we use the support forms for $C$ and $D$ in the given order).
\end{proof}

The theorem justifies us in calling the pair $(\epsilon_C(C), \epsilon_C(\ZZ^d))$ the \emph{$\ZZ$-normal form} of $C$. It can be computed in finitely many steps as we will now discuss.

Let $x_1,\dots,x_n$ be the Hilbert basis of $C$. Then $\epsilon_C(x_1),\dots,\epsilon_C(x_n))$ generate $\epsilon_C(C)$, and the matrix $S$ given by 
\begin{equation}
S_{ij}=\sigma_j(x_i), \qquad i=1,\dots,n, j=1,\dots,s,\label{type}
\end{equation}
determines the isomorphism type of $C$, and the isomorphism type of $C$ determines $S$ up to the order of the rows and columns. By the \emph{canonical form} of $S$ we denote the lexicographically greatest matrix that one can obtain by permutations of rows and columns from $S$ where we compare matrices of the same format lexicographically as follows: 
\begin{enumerate}
\item a row is lexicographically greater than another row if it is lexicographically greater under the comparison of coordinates from left to right;
\item a matrix is lexicographically greater than another matrix if it is lexicographically greater under the comparison of rows from top to bottom.
\end{enumerate}
For this choice the unit matrix is its own canonical form. We let 
$$
\type_\ZZ C
$$ 
denote the canonical form of $S$.

\begin{theorem}\label{type_iso}
Let $C$ and $D$ be cones. Then the following are equivalent:
\begin{enumerate}
\item $C$ and $D$ are $\ZZ$-isomorphic;
\item $C^*$ and $D^*$ are $\ZZ$-isomorphic;
\item $\type_\ZZ C = \type_\ZZ D$.
\end{enumerate}
\end{theorem}

Despite the symmetry of the statements (1) and (2) in the theorem, the construction of $\type_\ZZ C$ is not symmetric in $C$ and $C^*$: we use the Hilbert basis of $C$, but only the extreme integral generators of $C^*$. In order to achieve symmetry, one could use $\Hilb C^*$, but there is no need for this additional complication.

Our observations so far also allow us to identify the $\ZZ$-automorphism group of $C$ with a finite group of permutations of a purely combinatorial object. In the following $\cS_n$ denotes the permutation group of $\{1,\dots,n \}$.

\begin{theorem}\label{aut_perm}
Let $x_1,\dots,x_n$ be the Hilbert basis of $C$ and $\sigma_1,\dots,\sigma_s$ the support forms.. Then $\Aut_\ZZ C$ can be identified with the group of permutations $\Pi\in \cS_n$ for which there exists a permutation $\Sigma\in \cS_s$ such that 
\begin{equation}\label{GxLF1}
\langle x_{\Pi(i)}, \sigma_{\Sigma(j)}\rangle = \langle x_i, \sigma_j\rangle,,\qquad i=1,\dots,n,\ j=1,\dots,s.
\end{equation}
\end{theorem}

This theorem reduces the computation of the isomorphism type and the $\ZZ$-automorph\-ism group of a cone $C$ to finitely many steps. Suppose $C$ is defined by a system of generators. Then we have to compute
\begin{enumerate}
\item the support forms of $C$,
\item $\Hilb C$,
\item $\type_\ZZ C$ as the canonical from of the matrix $S$, and/or
\item the group of permutations $\Pi$ in Theorem \ref{aut_perm}.
\end{enumerate} 
Each of these tasks can be expensive, and the Hilbert basis is often the most critical step. Fortunately there is a good chance to get away without it, and it must be avoided if $\Aut_\ZZ C$ is to be exploited in the computation of $\Hilb C$.

\subsection{Using only the extreme rays}

Let $y_1,\dots,y_m$ be the extreme integral generators of $C$. Instead of the full matrix $S$ in \eqref{type} we can consider
\begin{equation}
E=(\sigma_j(y_i): i=1,\dots,m, j=1,\dots,s).\label{etype}
\end{equation}
It is a row selection of $S$ since the extreme integral generators belong to the Hilbert basis. We denote its canonical form by
$$
\etype_\ZZ C.
$$

\begin{proposition}\label{etype_iso}
Let $C$ and $D$ be cones in $\RR^d$.
\begin{enumerate}
\item If $C$ and $D$ are $\ZZ$-isomorphic, then $\etype_\ZZ C = \etype_\ZZ D$.
\item The converse holds if both the extreme integral generators of $C$ and those of $D$, respectively, generate $\ZZ^d$.
\item The converse holds as well if both the support forms of $C$ and those of $D$, respectively, generate $\ZZ^d$.
\end{enumerate}
\end{proposition}

\begin{proof}
(1) is obvious. For (2) we let $M$ by the monoid generated by the extreme integral generators of $C$ and $N$ the corresponding monoid for $D$. The detour via the standard embeddings shows that they are isomorphic monoids. The isomorphism extends both to an isomorphism of the cones they generate, namely $C$ and $D$, as well as to an isomorphism of the groups generated by them, which is $\ZZ^d$ in both cases. By Theorem \ref{basic} the cones $C$ and $D$ are $\ZZ$-isomorphic

Note that $\etype_\ZZ C=\etype_\ZZ D$ implies $\etype_\ZZ C^*=\etype_\ZZ D^*$ by transposition of matrices. Thus $C^*$ and $D^*$ are $\ZZ$-isomorphic by (2) and then $C$ and $D$ are $\ZZ$-isomorphic by Theorem \ref{type_iso}. This shows (3).
\end{proof}

\begin{corollary}\label{aut_ext}
Let $y_1,\dots,y_m$ be the extreme integral generators of $C$ and $\sigma_1,\dots,\sigma_s$ its support forms. Suppose that $y_1,\dots,y_m$ or $\sigma_1,\dots,\sigma_s$ generate $\ZZ^r$. Then
$\Aut_\ZZ C$ can be identified with the group of permutations $\Pi\in \cS_m$ for which there exists a permutation $\Sigma\in \cS_s$ such that 
\begin{equation}\label{GxLF2}
\langle y_{\Pi(i)}, \sigma_{\Sigma(j)}\rangle = \langle y_i, \sigma_j\rangle, \qquad i=1,\dots,m, \ j=1,\dots,s.
\end{equation} 
\end{corollary}

\begin{remark}
(a) Classical examples of cones $C$ and $D$ with $\etype_\ZZ C=\etype_\ZZ D$, but $\type_\ZZ C\neq \type_\ZZ D$, can be derived from ``empty'' simplices in $\RR^3$. These have been classified by White; see \cite[2.55]{BrGu}. As an explicit case we take the cones $C$ and $D$ 	generated by the row  vectors of the following two arrays:
$$
\begin{matrix}
0&0&0&1\\
0&1&0&1\\
0&0&1&1\\
5&1&1&1\\
\end{matrix}
\qquad\qquad
\begin{matrix}
0&0&0&1\\
0&1&0&1\\
0&0&1&1\\
5&2&1&1\\
\end{matrix}
$$ 
In both cases $\etype_\ZZ$ is $5E_4$ where $E_4$ is the $4\times 4$ unit matrix. But for $C$ (on the left) the Hilbert basis elements have two pairs of equal values under the $4$ support forms, whereas the Hilbert basis elements of $D$ have values $1,2,3,4$ under the support forms. Moreover, $\Aut_\ZZ C$ is the dihedral group $D_4$ of order $8$,  whereas $\Aut_\ZZ D\iso \ZZ_4$.

(b) The assumptions about $C$ and $C^*$ in Corollary \ref{aut_ext} are not equivalent: it is possible that $\sigma_1,\dots,\sigma_s$ generate $\ZZ^r$, whereas $y_1,\dots,y_m$ fail to do this (or vice versa). As a simple example one can take $C$ with the extreme rays $(0,0,1)$, $(0,1,1)$, $(2,0,1)$ and $(2,1,1)$.
\end{remark}

\subsection{Another approach to normal forms and automorphism groups}\label{another}

There is another approach to isomorphism classes and automorphism groups introduced by Bremner et al.\ in \cite{Bremner}. Let $v_1,\dots,v_n$ be vectors in $\RR^d$ generating $\RR^d$ as a vector space. For $i=1,\dots,n$ we form the symmetric $d\times d$-matrix given by
$$
(M_i)_{jk}=v_{ij}v_{ik},\qquad j,k=1,\dots,d,
$$
and set
$$
Q=\sum_{i=1}^{n} M_i. 
$$
It is not hard to see that $Q$ defines a positively definite quadratic form on $\RR^d$. Let $R=Q^{-1}$, and set
\begin{equation}
w_{ij}= v_i^TRv_j\label{Rdef}
\end{equation}
where we consider vectors as $d\times 1$-matrices and $T$ denotes transposition. By \cite[Prop. 3.1]{Bremner} the automorphisms of $\RR^d$ that permute $v_1,\dots,v_n$ correspond bijectively to the permutations of $\Pi\in \cS_n$ that satisfy
\begin{equation}\label{Gens1}
w_{ij}=w_{\Pi(i),\Pi(j)},\qquad i,j=1,\dots,n.
\end{equation}
This approach can be applied to extreme integral generators or the Hilbert basis. 

\subsection{Graded cones}\label{grading}

For enumerative tasks one often has to consider cones with a grading, i.e., an integral linear form $\gamma$ on $\RR^d$ such that $\gamma(x)>0$ for all $x\in C$, $x\ne 0$. Then only isomorphisms or automorphisms are of interest that respect the grading. The matrices $S$ and $E$ in \eqref{type} and \eqref{etype} get an extra column which is required to be left fix by passage to the canonical form of the matrix. In other words, it must be left fix by the  column permutations in Corollary \ref{aut_perm} and Corollary \ref{aut_ext}.

If the computation of isomorphisms and automorphisms exchanges $C$ and $C^*$, then the extra columns become extra rows, presenting fixed points, to be left invariant in the computation of canonical forms or the permutation groups  and the row permutations in Corollary \ref{aut_perm} and Corollary \ref{aut_ext}.

The grading is a \emph{special linear form} on $C$ and a \emph{special generator} of $C^*$.

\subsection{Rational polyhedra}

As pointed out in Section \ref{prelim}, computations for polyhedra $P\subset\RR^{d}$ are done in $\cone(P)\subset\RR^{d+1}$ and then restricted to the hyperplane $\cH(P)$ on which the dehomogenization $\delta$ has value $1$. This principle is applied to isomorphisms and automorphisms as well: we must additionally require that the automorphisms and isomorphisms leave the hyperplane $\cH(P)$ stable.

The dehomogenization is treated in the same way as the grading in Section \ref{grading}: it is a special linear form on $C$ and a special generator of $C^*$.

Additionally we may have a grading: an integral linear form on $\RR^d$ defines a grading on the polyhedron $P$ if it takes only positive values on the nonzero elements of $\rec P$. If we want to compute isomorphisms or automorphisms of graded polyhedra, then we must work with two special linear forms or generators.

\section{Other types of automorphisms}\label{other} 

\subsection{Rational and algebraic automorphisms}

Let $C\subset \RR^d$ be a cone. The automorphism group of $C$, i.e., the group of all $\RR$-linear automorphisms of $\RR C$ that map $C$ onto itself, is not finite, unless $C=0$. In fact, $\RR_+\subset \Aut_\RR C$ in a natural way. This observation remains true if we replace $\RR$ by a subfield $K$, such as $\QQ$ or a real algebraic number field, and $C$ by a cone generated by vectors with coordinates in $K$. In particular, $\Aut_K C$ cannot be computed as a subgroup of a finite group of permutations. (For a discussion of the full group of automorphisms we refer the reader to \cite{Bremner}.)

There is no natural replacement of the extreme integral generators if discrete structures are not involved. However, if $C$ is the cone over a polytope, then every extreme ray of $C$ has a distinguished point, namely the vertex of the polytope that is contained in it. This allows us to compute the $K$-automorphisms of a polytope, provided we can compute in $K$, and this is the case if $K=\QQ$ or $K$ is a real algebraic number field. We mention algebraic number fields since Normaliz can compute in them.

\subsection{Euclidean automorphisms}

A \emph{Euclidean automorphism} (or rigid motion) of a cone $C$ is a distance preserving (necessarily linear) automorphism of $\RR C$ that maps $C$ onto itself. In this case there is a distinguished set of points on the extreme rays, namely the points of Euclidean distance $1$ from the origin: a Euclidean automorphism of $C$ permutes them, and is uniquely determined by this permutation. But now a new difficulty arises: the points of distance $1$ usually have coordinates outside $\QQ$ or the given algebraic number field. One would have to adjoin a potentially large number of square roots in order to make the points of distance $1$ defined over the field of reference.

When considering polytopes we must be careful: a Euclidean motion of the hyperplane $\cH(P)$, does in general not extend to a Euclidean automorphism of the cone over the polytope. But this is not an obstruction to computability: we must find all rational automorphisms of the polytope that preserve the norms 
$$
\Vert v_i-v_j\Vert, \qquad i,j=1,\dots,n,
$$
for the vertices $v_1,\dots,v_n$ of our polytope. In order to stay in the field of rational numbers or in an algebraic number field, it is better to use the squares: we search all permutations $\Pi\in \cS_n$ that satisfy 
\begin{equation}\label{Gens2}
\Vert v_i-v_j\Vert^2 = \Vert v_{\Pi(i)}-v_{\Pi(j)}\Vert^2, \qquad i,j=1,\dots,n.
\end{equation}
Note that the hyperplane of the polytope is automatically preserved if we permute the vertices of the polytope.

There is one critical aspect regarding coordinate transformations: they must preserve Euclidean distances if Euclidean automorphisms are to be computed. This cannot be guaranteed  by Normaliz' coordinate transformations. Therefore Euclidean automorphism groups can only be computed if the input defines a polytope on the nose. The passage to a quotient makes no sense.

\subsection{Combinatorial automorphisms}

The combinatorial automorphisms of a polyhedron are the bijective maps of the set of faces to itself that preserve the partial order by inclusion. This is an abuse of terminology since combinatorial automorphisms of $P$ need not be automorphisms of $P$. It is not hard to see that the combinatorial automorphisms of a polyhedron $P$ can be identified with the combinatorial automorphisms of $\cC(P)$ that restrict to combinatorial automorphisms of the recession cone and therefore permute the faces of the polyhedron.

Every face of a polyhedron is the intersection of the facets in which it is contained, and every face of a pointed cone is spanned by the extreme rays in it. Let $x_1,\dots,x_n$ be the extreme rays of $C$ and $F_1,\dots,F_s$ its facets. Then we set $\delta_{ij}=1$ if $x_i\in F_j$, and $\delta_{ij}=0$ else. For the group of combinatorial automorphisms we must find all permutations $\Pi\in\cS_n$ for which there exists a permutation $\Sigma\in \cS_s$ such that
\begin{equation}\label{GxLF3}
\delta_{\Pi(i),\Sigma(j)}= \delta_{ij}, \qquad i=1,\dots,n, \ j=1,\dots,s,
\end{equation}
with the additional requirement in the case of a polyhedron that the incidence of extreme rays with $\cH(P)$ is preserved by $\Pi$. (It makes no sense for combinatorial automorphisms to respect the grading.)

\subsection{Non-intrinsically defined automorphisms}

It is often desirable to compute an automorphism group, or at least a subgroup, from partial information about the cone or polyhedron, namely using only the defining data. Usually these are extreme rays or, dually, support forms of the cone (over the polyhedron), potentially augmented by a grading. Normaliz therefore has a computation goal ``input automorphisms''. These are rational automorphisms preserving the set of generators, and additionally the grading, the dehomogenization or both. Before computing the input automorphisms, Normaliz prepares the input data as far as possible without the dualization of the cone. Then we are exactly in the situation of Section \ref{another}. With notation introduced there, we must find all permutations of the input vectors $v_1,\dots,v_n$ that satisfy the condition
\begin{equation}\label{Gens3}
w_{ij}=w_{\Pi(i),\Pi(j)},\qquad i,j=1,\dots,n.
\end{equation}
where $w_{ij}= v_i^TRv_j$.

An even coarser type is presented by ``ambient automorphisms'', namely coordinate \emph{permutations} of the ambient space that  preserve the input data. Suppose $C$ is defined by vectors $v_1,\dots,v_n\in\RR^d$. Then we search all coordinate permutations $\Pi\in \cS_n$ for which there is a  permutation $\Sigma\in \cS_d$  such that
\begin{equation}\label{GxLF4}
v_{\Pi(i),\Sigma(j)}=v_{ij},  \qquad i=1,\dots,n,\ j=1,\dots,d.
\end{equation} 
Additionally we require that the dehomogenization or the grading or both  are preserved.

Let $G$ be the group of ambient automorphisms. The elements of $G$ induce integral as well as Euclidean automorphisms of the cone or polyhedron, but in general the natural map from $G$ to the group $\Aut_\ZZ C$ is not surjective, and even injectivity is not guaranteed. 

\section{Implementation in Normaliz}\label{impl}

The automorphism groups that can be computed by Normaliz are realized as permutation groups. There are two types:
\begin{enumerate}
\item pairs of permutations of vectors and linear forms preserving the evaluation of the canonical bilinear form on $\RR^d\times (\RR^d)^*$, as in \eqref{GxLF1}, \eqref{GxLF2}, \eqref{GxLF3}, \eqref{GxLF4};
\item permutations of vectors preserving the evaluation of a symmetric bilinear form on $\RR^d$, as in  \eqref{Gens1}, \eqref{Gens2}, \eqref{Gens3}.
\end{enumerate}

Both problems can be considered as the computation of the automorphism group of a weighted graph:
\begin{enumerate}
\item For (1) above we choose the complete bipartite graph whose vertices are presented by the vectors in the first partition and the linear forms in the second. The weight associated to the edge connecting a vector $v$ and a linear form $\lambda$ is $\lambda(v)$.

\item For (2) we choose the complete graph whose vertices are presented by the vectors, augmented by edges connecting a vertex to itself. The weight of the edge connecting vectors $v$ and $w$ is $\beta(v,w)$ where $\beta$ denotes the symmetric bilinear form.
\end{enumerate}

The package \verb|nauty| by McKay and Piperno computes automorphism groups and canonical forms of undirected graphs. For the application to more general weighted graphs, these are replaced by a tower of  graphs with weights $0$ and $1$ only. The encoding is described in the nauty manual.

The ``raw'' weights $\lambda(v)$ and $\beta(v,w)$ above can be very large numbers. This is often true for the evaluation of the bilinear form defined by the matrix $R$ appearing in \eqref{Rdef}. Therefore it is better to sort the weights first and to replace each weight by its index in the sorting order, keeping the sorted weights in an extra vector. 

nauty allows the subdivision of the vertices of the graph into partitions left stable under the permutations that represent the automorphisms. This subdivision is not only necessary to separate the layers in the tower, but also for our bipartite graphs, and the special linear forms and vectors that must be accommodated as well. Each of the latter constitutes a partition in every layer so that they are fixed by the automorphisms.

nauty not only computes the automorphism group as a permutation group, but also returns the orbits of the vertices of the graph that we can interpret as permutations of vectors and/or linear forms, for example as permutations of extreme integral generators and support forms. Another important result is the canonical order of the vertices, on which we can base the computation of isomorphism types.  

The computation goals of Normaliz are selected by options in the input file or on the command line. The goal  \verb|Automorphisms| asks for integral automorphisms if the input is defined over the rational numbers and for algebraic automorphisms in the case of real algebraic number fields. All other types of automorphism groups are named exactly as in this note.

Only the computation of integral automorphisms may require several attempts. Normaliz first tries the computation based only on the extreme rays or the support forms, choosing the smaller cardinality. The matrices of the generators of the automorphism group are then checked for being defined over $\ZZ$ and having determinant $\pm 1$. If this is not the case, the other set of ``generators'' is used, and if this fails as well, then the Hilbert basis is computed.

\section{Two classical examples}\label{examples}
We illustrate the computational potential of Normaliz by two classical examples, the icosahedron and the linear ordering polytope.

\subsection{The icosahedron}
The icosahedron is one of the Platonic solids, a regular polytope of dimension $3$ defined over $\QQ[\sqrt 5]$, with $12$ vertices and $20$ facets. We compute the Euclidean automorphisms and get the following output:
\begin{footnotesize}
\begin{verbatim}
Euclidean automorphism group of order 120 (possibly approximation if very large)
Integrality not known
************************************************************************
3 permutations of 12 vertices of polyhedron

Perm 1: 1 2 4 3 7 8 5 6 10 9 11 12
Perm 2: 1 3 2 5 4 6 7 9 8 11 10 12
Perm 3: 2 1 3 4 6 5 8 7 9 10 12 11

Cycle decompositions 

Perm 1: (3 4) (5 7) (6 8) (9 10) --
Perm 2: (2 3) (4 5) (8 9) (10 11) --
Perm 3: (1 2) (5 6) (7 8) (11 12) --

1 orbits of vertices of polyhedron

Orbit 1 , length 12:  1 2 3 4 5 6 7 8 9 10 11 12

************************************************************************
3 permutations of 20 support hyperplanes

Perm 1: 2 1 5 6 3 4 7 8 11 12 9 10 13 14 17 18 15 16 20 19
...

Cycle decompositions 

Perm 1: (1 2) (3 5) (4 6) (9 11) (10 12) (15 17) (16 18) (19 20) --
...

1 orbits of support hyperplanes

Orbit 1 , length 20:  1 2 3 4 5 6 7 8 9 10 11 12 13 14 15 16 17 18 19 20
\end{verbatim}
\end{footnotesize}
The permutations represent a system of generators of the automorphism group. It is easy to see that the group of combinatorial automorphisms cannot have order $>120$. Therefore all intrinsically defined automorphism groups of this classical polyhedron have the same order $120$.

The $120$-cell and the $600$-cell, regular polytopes of dimension $4$ and both defined over $\QQ[\sqrt 5]$, have the Coxeter group $H_4$ as their Euclidean automorphism group, and at least the order $14,400$ can be verified by Normaliz.

\subsection{The linear ordering polytope} \def\LO{\operatorname{LO}} 
Let $I=\{1,\dots,n\}$, $n\ge 3$. (The case $n=2$ is trivial and special.) A relation $R$ on $I$ is a subset of $I\times I$, and $R$ can be encoded by its incidence vector: it is a $0$-$1$-vector $\iota_R$ whose components are labeled by the elements $(x,y)$ of $I\times I$, and $\iota_R(x,y)=1$ if $(x,y)\in R$ and $\iota_R(x,y)=0$ else. The \emph{linear  ordering polytope} $\LO_n$ is the convex hull of the incidence vectors of the linear (or total) orders on $I$. It has been explored extensively; for example, see the book \cite{MaRe} by  Mart\'i and Reinelt.

The group $\cS_n$ of permutations of $I$ acts on the set of all incidence vectors by \emph{relabeling}:  $\iota_{\pi(R)}(x,y)=\iota_R(\pi(x),\pi(y))$. This action is a permutation of coordinates that restricts to the linear ordering polytope. There is a further permutation of coordinates with this property, the \emph{duality}: $\iota_{R^\vee}(x,y) = \iota_R(y,x)$. Evidently dualization is different from all relabelings and commutes with them. Therefore $\Aut_\ZZ \LO_n$ contains a copy of $\ZZ_2\times \cS_n$. All these automorphisms are Euclidean as well.

Expecting that $\ZZ_2\times \cS_6$ is the full group of automorphisms of $\LO_6$, we apply Normaliz to it---and get a surprising result: the group of integral automorphisms has order $10,080 = 2|\cS_{7}|$. It does not change if ``integral'' is replaced by the more general ``rational'' or ``combinatorial''. But if we replace it by ``Euclidean'', then the order goes down to $1,440 = 2|\cS_6|$, and the computation confirms that above we have found all Euclidean automorphisms of $\LO_6$. Computations for $n=3, 4, 5, 7$ show the same pattern.

It is actually known that the group of combinatorial automorphisms of $\LO_n$ is isomorphic to $\ZZ_2\times \cS_{n+1}$. See Fiorini \cite{Fio}. We complement this result by proving that $\ZZ_2\times \cS_n$ is the Euclidean automorphism group. To this end we must understand the action of $\ZZ_2\times \cS_{n+1}$ on $\LO_n$. We  follow Katthän \cite{Katth}, using his conventions and notation.

The first step is a change of the ambient space: an order relation $R$ is completely determined by the pairs $(x,y)\in R$, $x<y$. The corresponding orthogonal projection of $\RR^{n^2}$ to $\RR^{\binom{n}{2}}$ maps the linear ordering polytope to an integrally isomorphic copy, and even the Euclidean structure remains unchanged: all distances between vertices change by the factor $1/\sqrt 2$. We identify $\LO_n$ with its projection to $\RR^{\binom{n}{2}}$. It is easy to check that the vertices of $\LO_n$ generate the full lattice $\ZZ^{\binom{n}{2}}$, as an affine lattice as well as a subgroup of $\ZZ^{\binom{n}{2}}$.

To make relabeling and dualization linear---so far they are only affine linear---we apply the affine linear bijective map
$$
\Psi:\RR^{\binom{n}{2}} \to \RR^{\binom{n}{2}},\qquad \Psi(v) =2v - \mathbf{1},
$$
where $\mathbf{1}$ is the vector with all coordinates equal to $1$. Moreover we set $\tilde\LO_n=\Psi(\LO_n)$.

\begin{lemma}\label{tilde}
Lat $\phi$ be a map of $\RR^{\binom{n}{2}}$ to itself, and set $\tilde\phi=\Psi\phi\Psi^{-1}$.
\begin{enumerate}
\item $\phi$ is a Euclidean automorphism of $\LO_n$ if and only $\tilde\phi$ is a Euclidean automorphism of $\tilde\LO_n$.
\item $\phi(\mathbf 1) = \mathbf 1\iff\tilde\phi(\mathbf 1)=\mathbf 1$.
\item Suppose that $\tilde\phi$ is a Euclidean automorphism of $\tilde\LO_n$ with  $\tilde\phi(\mathbf 1)=\mathbf 1$. Then $\tilde\phi$ permutes the unit vectors in $\RR^{\binom{n}{2}}$.
\end{enumerate}
\end{lemma}

\begin{proof}
(1) and (2) are obvious. For (3) we start from $\tilde\phi$ and set $\phi=\Psi^{-1}\tilde\phi=\Psi$. Then $\phi$ is a Euclidean automorphism of $\LO_n$ with $\phi(\mathbf 1)=\mathbf 1$ by (1) and (2). Moreover, $\phi(0)=0$ since $0$ is the uniquely determined vertex of $\LO_n$ with maximum distance from $\mathbf 1$. Hence $\phi$ is linear (and not only affine linear). Clearly $\phi$ extends to a Euclidean automorphism of $\RR^{\binom{n}{2}}$, and it is integral because the vertices of $\LO_n$ generate the lattice $\ZZ^{\binom{n}{2}}$. An integral Euclidean automorphism permutes the unit vectors up to sign, and if it fixes $\mathbf 1$, it must permute the unit vectors themselves. That also $\tilde\phi$ permutes the unit vectors is now an easy computation.
\end{proof}

The next transition is from orderings to permutations. For $\pi\in \cS_n$ we define the relation $R_\pi$ by $(x,y)\in R_\pi$ if and only if $\pi(x)>\pi(y)$. It follows that the polytope $\tilde\LO_n$ spanned by the vectors
$$
(\tilde k_{ij}(\pi): 1\le i < j\le n), \qquad \pi\in \cS_n,
$$
where
$$
\tilde k_{ij}(\pi)=\begin{cases}
1&\text{if }\pi(i)>\pi(j),\\
-1&\text{if }\pi(i)<\pi(j).
\end{cases}
$$
Setting $\tilde k_{ij}=-\tilde k_{ij}$ for $i>j$ yields an easy description of relabeling by a permutation $\sigma\in \cS_n$: it amounts to the transformation 
$$
\sigma(\tilde k_{ij}) = \tilde k_{\sigma(i)\sigma(j)}.
$$
Dualization is multiplication by $-1$.

Relabeling makes $\cS_n$ act on the vector subspace $W$ of the space of real valued functions on $\cS_n$ generated by the functions $\tilde k_{ij}$. Katthän's crucial insight is to identify $W$ (which is $\tilde U_{Inv}$ in \cite{Katth}) with the second exterior power $\bigwedge^2 V$ of $V=\RR^n$ and to observe that $W$ and $\bigwedge^2 V$ are isomorphic representations of the group $\cS_n$ if we let $\cS_n$ act on $V$ by the relabeling of the unit vectors, $\sigma(e_i)=e_{\sigma(i)}$, and extend this action to $\bigwedge^2 V$ in the natural way. Under the action of $\cS_n$, the vector space $V$ is not irreducible: it has a $1$-dimensional subspace of invariants spanned by $e_1+\dots+e_n$. The complementary $\cS_n$-subspace is the orthogonal complement of the subspace of invariants since the action of $\cS_n$ preserves the standard scalar product. In it we choose the basis $e_i-e_n$, $i=1,\dots,n-1$. The splitting of $V$ induces a splitting of $\bigwedge^2 V$ into two summands, one of which has the basis $(e_i-e_n)\wedge (e_j-e_n)=e_i\wedge e_j - e_i\wedge e_n + e_j\wedge e_n$, $1\le i < j< n$. The transfer to $W$ shows that the subspace generated by the functions
$$
\tilde w_{ij}=\tilde k_{ij}-\tilde k_{in}+\tilde k_{jn} , \qquad 1\le i < j <n,
$$
is closed under relabeling (and it is obviously closed under dualization). 

Let $P'$ be the polytope spanned by the vectors
$$
(\tilde w_{ij}(\pi): 1\le i < j<n), \qquad \pi\in \cS_n.
$$
One checks that $\tilde w_{ij}(\zeta \pi) = \tilde w_{ij}(\pi)$ for the cyclic permutation $\zeta$, $\zeta(k) = k+1 \mod n$. Since $\zeta^m\pi(n)=n$ for suitable $m$ and the subgroup of $\cS_n$ formed by the permutation fixing $n$ can be identified with $\cS_{n-1}$, one gets the desired identification of $P'$ with $\tilde\LO_{n-1}$ plus a description of the action of $\cS_n$ on $\tilde\LO_{n-1}$.

\begin{theorem}
For all $n\ge 3$ the group of Euclidean automorphisms of $\tilde\LO_n$ is $\ZZ_2\times \cS_n$ acting by relabeling and dualization.
\end{theorem} 

\begin{proof} 
We know already that $\ZZ_2\times \cS_n$ acts by Euclidean automorphisms. Assume that $\phi\in\ZZ_2\times \cS_{n+1}$ with $\phi\notin \ZZ_2\times\cS_n$, does this as well. Since dualization is the point reflection at the midpoint of $\LO_n$, we can assume $\phi\in \cS_{n+1}\setminus \cS_n$. Since $\cS_n$ acts transitively on the vertices of $\LO_n$, we can even assume that $\phi(\mathbf 1)=\mathbf 1$. By Lemma \ref{tilde} we pass to $\tilde \LO_n$ and get an automorphism if $\tilde\LO_n$ that permutes the unit vectors, equivalently, the coordinates,  in $\RR^{\binom{n}{2}}$.
	
Now $\tilde\LO_n$ can be identified with $P'$ above after replacing  $n$ by $n-1$. Our assumption on $\phi$ is $\phi(n)\neq n$. Since $n\ge 4$, we find $i,j$ with $1\le i<j < n$ and $\phi(i),\phi(j)< n$. Relabeling by $\phi$ does not transform $\tilde w_{ij}$ into another coordinate function $\tilde w_{uv}$. This is a contradiction.
\end{proof}

One can use Normaliz to explore further properties of the polytopes $\LO_n$ for $n\le 7$. The volumes of these polytopes can be computed by the descent algorithm with the exploitation of isomorphism types of faces. For $n\le 6$ the Ehrhart series is computable as well.  It is not known whether all linear order polytopes are normal in the sense of \cite{BrGu}. Normaliz confirms normality rather quickly for $n\le 6$. A brute force application of Normaliz in \cite{Rudy} has verified it for $n=7$ as well. One of the next releases of Normaliz will exploit the action of the automorphism group for this computation and accelerate it. 

The facets of $\LO_n$ are known only for $n\le 7$ (for $n=8$ one has a lower bound of their number and of the number of orbits). Normaliz can determine them, including their orbits under the actions of $\ZZ_2\times \cS_{n+1}$ and $\ZZ_2\times\cS_n$. For $n=7$ the computation confirms (and is confirmed by) Tables 6.1 and 6.2 in \cite{MaRe}.  The $\ZZ_2\times \cS_{n+1}$-orbits are called ``$P_{\LO}^n $-classes'' in \cite{MaRe}. Table 6.1 shows the $19$  $\ZZ_2\times\cS_n$-orbits for $n=7$, of which $8$ decompose into $2$  $\cS_n$-orbits.

\section{Acknowledgment}

We are grateful to Takayuki Hibi with whom the first ideas towards this note were discussed.

Our thanks go to Lukas Katthän for providing the reference \cite{Fio} and for critically reading this note.

The use of nauty in Normaliz would not have been possible without an intensive consultation of Brendan McKay.

\end{document}